\pdfoutput=1
\RequirePackage{ifpdf}
\ifpdf 
\documentclass[pdftex]{sigma}
\else
\documentclass{sigma}
\fi

\numberwithin{equation}{section}

\newtheorem{Theorem}{Theorem}[section]
\newtheorem{Lemma}[Theorem]{Lemma}
 { \theoremstyle{definition}
\newtheorem{Example}[Theorem]{Example}
\newtheorem{Remark}[Theorem]{Remark} }

\newcommand{\cch}{\Omega}
\newcommand{\oot}{\bar\otimes}

\newcommand{\C}{{\mathcal C}}

\newcommand{\D}{{\mathcal D}}

\newcommand{\M}{{\mathcal M}}

\newcommand{\1}{\textbf{1}}

\newcommand\id{\operatorname{id}}

\newcommand\ad{\operatorname{ad}}

\newcommand\Fun{\operatorname{Fun}}

\newcommand\FPdim{\operatorname{FPdim}}

\begin{document}


\newcommand{\arXivNumber}{1608.04435}

\renewcommand{\PaperNumber}{042}

\FirstPageHeading

\ShortArticleName{On the Equivalence of Module Categories over a Group-Theoretical Fusion Category}

\ArticleName{On the Equivalence of Module Categories\\ over a Group-Theoretical Fusion Category}

\Author{Sonia NATALE}

\AuthorNameForHeading{S.~Natale}

\Address{Facultad de Matem\'atica, Astronom\'{\i}a, F\'{\i}sica y Computaci\'on, Universidad Nacional de C\'ordoba,\\ CIEM-CONICET, C\'ordoba, Argentina}
\Email{\href{mailto:natale@famaf.unc.edu.ar}{natale@famaf.unc.edu.ar}}
\URLaddress{\url{http://www.famaf.unc.edu.ar/~natale/}}

\ArticleDates{Received April 28, 2017, in f\/inal form June 14, 2017; Published online June 17, 2017}

\Abstract{We give a necessary and suf\/f\/icient condition in terms of group cohomology for two indecomposable module categories over a group-theoretical fusion category ${\mathcal C}$ to be equivalent. This concludes the classif\/ication of such module categories.}


\Classification{18D10; 16T05}

\section{Introduction}

Throughout this paper we shall work over an algebraically closed f\/ield $k$ of characteristic zero. Let $\mathcal{C}$ be a fusion category over~$k$. The notion of a $\C$-module category provides a natural categorif\/ication of the notion of representation of a group. The problem of classifying module categories plays a fundamental role in the theory of tensor categories.

Two fusion categories $\C$ and $\D$ are called \emph{categorically Morita equivalent} if there exists an indecomposable $\mathcal{C}$-module category $\mathcal{M}$ such that $\mathcal{D}^{\rm op}$ is equivalent as a fusion category to the category $\operatorname{Fun}_\C(\M, \M)$ of $\C$-module endofunctors of~$\M$. This def\/ines an equivalence relation in the class of all fusion categories.

Recall that a fusion category $\C$ is called \emph{pointed} if every simple object of $\C$ is invertible. A~basic class of fusion categories consists of those which are categorically Morita equivalent to a pointed fusion category; a fusion category in this class is called \emph{group-theoretical}. Group-theoretical fusion categories can be described in terms of f\/inite groups and their cohomology.

The purpose of this note is to give a necessary and suf\/f\/icient condition in terms of group cohomology for two indecomposable module categories over a group-theoretical fusion category to be equivalent. For this, it is enough to solve the same problem for indecomposable module categories over pointed fusion categories.

Let $\C$ be a pointed fusion category. Then there exist a f\/inite group $G$ and a~3-cocycle~$\omega$ on~$G$ such that $\C \cong \C(G, \omega)$, where $\C(G, \omega)$ is the category of f\/inite-dimensional $G$-graded vector spaces with associativity constraint def\/ined by $\omega$ (see Section~\ref{cgomega} for a precise def\/inition). Let~$\M$ be an indecomposable right $\C$-module category. Then there exists a~subgroup~$H$ of~$G$ and a~2-cochain $\psi \in C^2(H, k^{\times})$ satisfying
\begin{gather}
\label{cond-alfa}d\psi = \omega\vert_{H \times H \times H},
\end{gather}
such that $\M$ is equivalent as a $\C$-module category to the category $\M_0(H, \psi)$ of left $A(H, \psi)$-modules in $\C$, where $A(H, \psi) = k_\psi H$ is the group algebra of $H$ with multiplication twisted by~$\psi$~\cite{ostrik}, \cite[Example~9.7.2]{egno}.

The main result of this paper is the following theorem.

\begin{Theorem}\label{main} Let $H, L$ be subgroups of $G$ and let $\psi \in C^2(H, k^{\times})$ and $\xi \in C^2(L, k^{\times})$ be $2$-cochains satisfying condition~\eqref{cond-alfa}. Then $\M_0(H, \psi)$ and $\M_0(L, \xi)$ are equivalent as $\C$-module categories if and only if there exists an element $g \in G$ such that $H = {}^gL$ and the class of the $2$-cocycle
\begin{gather}\label{cond-equiv} {\xi}^{-1}{\psi}^g \cch_g\vert_{L\times L}
\end{gather}
is trivial in $H^2(L, k^{\times})$.
\end{Theorem}

Here we use the notation ${}^g x = gxg^{-1}$ and $^gL = \{{}^g x\colon x\in L\}$. The 2-cochain ${\psi}^g \in C^2(L, k^{\times})$ is def\/ined by ${\psi}^g(g_1, g_2) = {\psi}({}^gg_1, {}^gg_2)$, for all $g_1, g_2 \in L$, and $\cch_g\colon G \times G \to k^{\times}$ is given by
\begin{gather*}\cch_g(g_1, g_2) = \frac{\omega({}^gg_1, {}^gg_2, g) \omega(g, g_1, g_2)}{\omega({}^gg_1, g, g_2)}.
\end{gather*}

Observe that \cite[Theorem 3.1]{ostrik} states that the indecomposable module categories considered in Theorem~\ref{main} are parameterized by conjugacy classes of pairs $(H, \psi)$. However, this conjugation relation is not described loc.\ cit.\ (compare also with \cite{nikshych} and \cite[Section~9.7]{egno}).

\looseness=-1 Consider for instance the case where $\C$ is the category of f\/inite-dimensional representations of the 8-dimensional Kac Paljutkin Hopf algebra. Then $\C$ is group-theoretical. In fact, $\C \cong \C(G, \omega, C, 1)$, where $G \cong D_8$ is a semidirect product of the group $L = \mathbb Z_2 \times \mathbb Z_2$ by $C = \mathbb Z_2$ and~$\omega$ is a certain (nontrivial) 3-cocycle on $G$~\cite{schauenburg}. Let $\xi$ represent a nontrivial cohomology class in $H^2(L, k^\times)$. According to the usual conjugation relation among pairs $(L, \psi)$, the result in \cite[Theorem~3.1]{ostrik} would imply that the pairs $(L, 1)$ and $(L, \xi)$, not being conjugated under the adjoint action of $G$, give rise to two inequivalent $\C$-module categories. These module categories both have rank one, whence they give rise to non-isomorphic f\/iber functors on $\C$. However, it follows from \cite[Theorem~4.8(1)]{ma-contemp} that the category $\C$ has a unique f\/iber functor up to isomorphism. In fact, in this example there exists $g \in G$ such that $\cch_g\vert_{L\times L}$ is a 2-cocycle cohomologous to~$\xi$. See Example~\ref{kp}.

Certainly, the condition given in Theorem \ref{main} and the usual conjugacy relation agree in the case where the 3-cocycle $\omega$ is trivial, and it reduces to the conjugation relation among subgroups when they happen to be cyclic.

As explained in Section~\ref{adj-action}, condition~\eqref{cond-equiv} is equivalent to the condition that $A(L, \xi)$ and ${}^gA(H, \psi)$ be isomorphic as algebras in~$\C$, where $\underline{G} \to \underline{\text{Aut}}_\otimes\C$, $g \mapsto {}^g( \ )$, is the adjoint action of~$G$ on~$\C$ (see Lemma~\ref{gdea}).

Theorem \ref{main} can be reformulated as follows.

\begin{Theorem} Two $\C$-module categories $\M_0(H, \psi)$ and $\M_0(L, \xi)$ are equivalent if and only if the algebras $A(H, \psi)$ and $A(L, \xi)$ are conjugated under the adjoint action of $G$ on $\C$.
\end{Theorem}

Theorem \ref{main} is proved in Section~\ref{demo}. Our proof relies on the fact that, as happens with group actions on vector spaces, the adjoint action of the group $G$ in the set of equivalence classes of $\C$-module categories is trivial (Lemma~\ref{adj-triv}). In the course of the proof we establish a relation between the 2-cocycle in~\eqref{cond-equiv} and a 2-cocycle attached to $g$, $\psi$ and $\xi$ in~\cite{ostrik} (Remark~\ref{rmk-alfag} and Lemma~\ref{rel-cociclos}).

We refer the reader to~\cite{egno} for the main notions on fusion categories and their module categories used throughout.

\section{Preliminaries and notation}\label{prels}

\subsection{}
Let $\mathcal{C}$ be a fusion category over $k$. A (\emph{right}) $\C$-\emph{module category} is a f\/inite semisimple $k$-linear abelian category $\mathcal{M}$ equipped with a bifunctor $\oot\colon \mathcal{M}\times \mathcal{C} \rightarrow\mathcal{M}$ and natural isomorphisms
\begin{gather*}\mu_{M, X,Y}\colon \ M \oot (X\otimes Y)\rightarrow (M\oot X)\oot Y,\qquad r_M\colon \ M \oot \textbf{1}\rightarrow M, \end{gather*}
$X, Y \in \C$, $M \in \mathcal M$, satisfying the following conditions:
\begin{gather} \label{modcat1} \mu_{M\oot X, Y, Z} \mu_{M, X, Y \otimes Z} ( \id_M \oot a_{X, Y, Z} ) =
 (\mu_{M, X, Y} \otimes \id_Z ) \mu_{M, X\otimes Y, Z},\\
 \label{modcat2} (r_M \otimes \id_Y ) \mu_{M, \1, Y} = \id_M \oot l_Y,
\end{gather}
for all $M \in \M$, $X, Y \in \C$, where $a\colon \otimes \circ (\otimes \times \id_\C) \to \otimes \circ (\id_\C \times \otimes)$ and $l\colon \textbf{1} \otimes ? \to \id_\C$, denote the associativity and left unit constraints in $\C$, respectively.

Let $A$ be an algebra in $\C$. Then the category $_A\C$ of left $A$-modules in $\C$ is a right $\C$-module category with action $\oot\colon _A\C \times \C \to \C_A$, given by $M\oot X = M\otimes X$ endowed with the left $A$-module structure $(m_M \otimes \id_X) a_{A, M, X}^{-1}\colon A\otimes (M\otimes X) \to M\otimes X$, where $m_M\colon A\otimes M \to M$ is the $A$-module structure in $M$. The associativity constraint of $_A\C$ is given by $a^{-1}_{M, X, Y}\colon M\oot (X \otimes Y) \to (M\oot X) \oot Y$, for all $M \in {} _A\C$, $X, Y \in \C$.

A $\C$-module functor $\M \to \M'$ between right $\C$-module categories $(\M, \oot)$ and $(\M', \oot')$ is a pair $(F, \zeta)$, where $F\colon \M\to \M'$ is a functor and $\zeta_{M, X}\colon F(M \oot X) \to F(M) \oot' X$ is a natural isomorphism satisfying
\begin{gather}\label{uno-z} (\zeta_{M, X} \otimes \id_Y ) \zeta_{M \oot X, Y} F(\mu_{M, X, Y}) = {\mu'}_{F(M), X, Y} \zeta_{M, X \otimes Y},\\
\label{dos-z} {r'}_{F(M)} \zeta_{M, \1} = F(r_M),
\end{gather}
for all $M \in \M$, $X, Y \in \C$.

Let $\M$ and $\M'$ be $\C$-module categories. An \emph{equivalence} of $\C$-module categories $\M \to \M'$ is a $\C$-module functor $(F, \zeta)\colon \M\to \M'$ such that $F$ is an equivalence of categories. If such an equivalence exists, $\M$ and $\M'$ are called \emph{equivalent $\C$-module categories}. A~$\C$-module category is called \emph{indecomposable} if it is not equivalent to a direct sum of two nontrivial $\mathcal{C}$-submodule categories.

Let $\M, \M'$ be indecomposable $\C$-module categories. Then $\operatorname{Fun}_\C(\M, \M)$ is a fusion category with tensor product given by composition of functors and the category $\operatorname{Fun}_\C(\M, \M')$ is an indecomposable module category over $\operatorname{Fun}_\C(\M, \M)$ in a natural way. If $A$ and $B$ are indecomposable algebras in $\C$ such that $\M \cong _A\!\C$ and $\M' \cong {}_B\C$, then $\operatorname{Fun}_\C(\M, \M)^{\rm op}$ is equivalent to the fusion category $_A\C_A$ of $(A, A)$-bimodules in $\C$ and there is an equivalence of $_A\C_A$-module categories $_B\C_A \cong \operatorname{Fun}_\C(\M, \M')$, where $_B\C_A$ is the category of $(B, A)$-bimodules in $\C$.

\subsection{}\label{action-mc} Let $\M$ be a $\C$-module category. Every tensor autoequivalence $\rho\colon \C \to \C$ induces a $\C$-module category structure $\M^\rho$ on $\M$ in the form $M \oot^\rho X = M\oot \rho(X)$, with associativity constraint
\begin{gather*}\mu^\rho_{M, X, Y} = \mu_{M, \rho(X) \otimes \rho(Y)} \big(\id_M \oot {\rho^2_{X, Y}}^{-1}\big)\colon \ M \oot \rho(X\otimes Y) \to (M\oot \rho(X)) \oot \rho(Y),\end{gather*} for all $M \in \M$, $X, Y \in \C$, where $\rho^2_{X, Y}\colon \rho(X) \otimes \rho(Y) \to \rho(X \otimes Y)$ is the monoidal structure of~$\rho$. See \cite[Section~3.2]{nikshych}.

Suppose that $A$ is an algebra in $\C$. Then $\rho(A)$ is an algebra in $\C$ with multiplication
\begin{gather*}m_{\rho(A)} = \rho(m_A) \rho^2_{A, A}\colon \ \rho(A) \otimes \rho(A) \to \rho(A).\end{gather*}
The functor $\rho$ induces an equivalence of $\C$-module categories $_{\rho(A)}\C \to (_A\C)^\rho$ with intertwining isomorphisms
\begin{gather*} {\rho^2_{M, X}}^{-1}\colon \ \rho(M \oot X) \to \rho(M) \oot^\rho X.\end{gather*}

\subsection{}\label{cgomega} Let $G$ be a f\/inite group.
Let $X$ be a $G$-module. Given an $n$-cochain $f \in C^n(G, X)$ (where $C^0(G, M) = M$), the coboundary of $f$ is the $(n+1)$-cochain $df = d^nf \in C^{n+1}(G, X)$ def\/ined by
\begin{gather*}d^nf(g_1, \dots, g_{n+1}) = g_1.f(g_2, \dots, g_{n+1}) + \sum_{i = 1}^{n}f(g_1, \dots, g_ig_{i+1}, \dots, g_n)\\
 \hphantom{d^nf(g_1, \dots, g_{n+1}) =}{} + (-1)^{n+1} f(g_1, \dots, g_n),\end{gather*}
for all $g_1, \dots, g_{n+1} \in G$. The kernel of $d^n$ is denoted $Z^n(G, M)$; an element of $Z^n(G, M)$ is an $n$-cocycle. We have $d^nd^{n-1} = 0$, for all $n \geq 1$. The $n$th cohomology group of $G$ with coef\/f\/icients in $M$ is $H^n(G, M) = Z^n(G, M)/d^{n-1}(C^{n-1}(G, M))$. We shall write $f \equiv f'$ when the cochains $f, f' \in C^n(G, k^\times)$ dif\/fer by a~coboundary.

We shall assume that every cochain $f$ is \emph{normalized}, that is, $f(g_1, \dots, g_n) = 1$, whenever one of the arguments $g_1, \dots, g_n$ is the identity. If $H$ is a subgroup of $G$ and $f \in C^n(H, k^\times)$, we shall indicate by~$f^g$ the $n$-cochain in $^{g^{-1}\!}H$ given by $f^g(h_1, \dots, h_n) = f({}^g h_1, \dots, {}^g h_n)$, $h_1, \dots, h_n \in H$.

Let $\omega \colon G \times G \times G \to k^\times$ be a 3-cocycle on $G$. Let $\C(G, \omega)$ denote the fusion category of f\/inite-dimensional $G$-graded vector spaces with associativity constraint def\/ined, for all $U, V, W\in \C(G, \omega)$, as
\begin{gather*}a_{X, Y, Z} ((u \otimes v)\otimes w) = \omega^{-1}(g_1, g_2, g_3) u \otimes (v\otimes w),\end{gather*}
for all homogeneous vectors $u \in U_{g_1}$, $v \in V_{g_2}$, $w \in W_{g_3}$, $g_1, g_2, g_3 \in G$. Any pointed fusion category is equivalent to a~category of the form $\C(G, \omega)$.

A fusion category $\C$ is called \emph{group-theoretical} if it is categorically Morita equivalent to a~pointed fusion category. Equivalently, $\C$ is group-theoretical if and only if there exist a f\/inite group $G$ and a $3$-cocycle $\omega\colon G \times G \times G \to k^{\times}$ such that~$\C$ is equivalent to the fusion category $\C(G, \omega, H, \psi) = _{A(H, \psi)\!}\C(G, \omega)_{A(H, \psi)}$, where $H$ is a subgroup of $G$ such that the class of $\omega\vert_{H\times H \times H}$ is trivial and $\psi\colon H \times H \to k^{\times}$ is a 2-cochain on~$H$ satisfying condition~\eqref{cond-alfa}.

Let $\C(G, \omega, H, \psi) \cong \C(G, \omega)^*_{\M_0(H, \psi)}$ be a group-theoretical fusion category. Then there is a bijective correspondence between equivalence classes of indecomposable $\C(G, \omega, H, \psi)$-module categories and equivalence classes of indecomposable $\C(G, \omega)$-module categories. This correspondence attaches to every indecomposable $\C(G, \omega)$-module category $\M$ the $\C(G, \omega, H, \psi)$-module category
\begin{gather*}\M(H, \psi) = \Fun_{\C(G, \omega)}(\M_0(H, \psi), \M). \end{gather*}

\section[Indecomposable module categories over $\C(G, \omega)$]{Indecomposable module categories over $\boldsymbol{\C(G, \omega)}$}\label{ptd-mc}

Throughout this section $G$ is a f\/inite group and $\omega \colon G \times G \times G \to k^\times$ is a 3-cocycle on $G$.

\subsection{}\label{adj-action} Let $g \in G$. Consider the 2-cochain $\cch_g\colon G \times G \to k^{\times}$ given by
\begin{gather*}\cch_g(g_1, g_2) = \frac{\omega({}^gg_1, {}^gg_2, g) \omega(g, g_1, g_2)}{\omega({}^gg_1, g, g_2)}.
\end{gather*}

For all $g \in G$ we have the relation
\begin{gather}\label{o-oy}d\cch_g = \frac{\omega}{\omega^g}.
\end{gather}

Let $\C = \C(G, \omega)$ and let $g \in G$. For every object $V$ of $\C$ let ${}^g V$ be the object of $\C$ such that $^gV = V$ as a vector space with $G$-grading def\/ined as $({}^gV)_x = V_{^g x}$, $x \in G$. For every $g \in G$, we have a functor $\ad_g \colon \C \to \C$, given by $\ad_g(V) = {}^gV$ and $\ad_g(f) = f$, for every object $V$ and morphism $f$ of $\C$.
Relation~\eqref{o-oy} implies that $\ad_g$ is a tensor functor with monoidal structure def\/ined by
\begin{gather*}\big(\ad_g^2\big)_{U, V}\colon \ {}^gU \otimes {}^gV \to {}^g (U \otimes V), \qquad \big(\ad_g^2\big)_{U, V} (u \otimes v) = \cch_g(h, h')^{-1} u\otimes v,\end{gather*}
for all $h, h' \in G$, and for all homogeneous vectors $u \in U_h$, $v \in V_{h'}$.

For every $g, g_1, g_2 \in G$, let $\gamma(g_1, g_2)\colon G \to k^\times$ be the map def\/ined in the form
\begin{gather*}
\gamma(g_1, g_2)(g) = \frac{\omega(g_1, g_2, g) \omega({}^{g_1g_2}g, g_1, g_2)}{\omega(g_1,{}^{g_2}g, g_2)}.
\end{gather*}
The following relation holds, for all $g_1, g_2 \in G$:
\begin{gather}\label{rel-omega}
\cch_{g_1g_2} = \cch_{g_1}^{g_2} \cch_{g_2} d \gamma(g_1, g_2) .
\end{gather}

In this way, $\ad\colon \underline G \to \underline{\text{Aut}}_\otimes\C$, $\ad(g) = \big(\ad_g, \ad_g^2\big)$, gives rise to an action by tensor autoequivalences of $G$ on $\C$ where, for every $g, x \in G$, $V \in \C(G, \omega)$, the monoidal isomorphisms ${\ad^2}_V\colon {}^g({}^{g'}V) \to {}^{gg'}V$ are given by
\begin{gather*}\ad^2_V(v) = \gamma(g, g') (x) v,\end{gather*} for all homogeneous vectors $v \in V_x$, $h \in G$. The equivariantization $\C^G$ with respect to this action is equivalent to the category of f\/inite-dimensional representations of the twisted quantum double $D^\omega G$ (see \cite[Lemma~6.3]{naidu}).

For each $g \in G$, and for each $\C$-module category $\M$, let $\M^g$ denote the module category induced by the functor $\ad_g$ as in Section~\ref{action-mc}. Recall that the action of $\C$ on $\M^g$ is def\/ined by $M\oot^g V = M\oot ({}^gV)$, for all objects~$V$ of~$\C$.

\begin{Lemma}\label{adj-triv} Let $g \in G$ and let $\M$ be a $\C$-module category. Then $\M^g \cong \M$ as $\C$-module categories.
\end{Lemma}

\begin{proof} For each $g \in G$, let $\{g\}$ denote the object of $\C$ such that $\{g\} = k$ with degree $g$. In what follows, by abuse of notation, we identify $\{g\} \otimes \{h\}$ and $\{gh\}$, $g, h \in G$, by means of the canonical isomorphisms of vector spaces.

Let $R_g\colon \M^g \to \M$ be the functor def\/ined by the right action of $\{g\}$: $R_g(M) = M \oot \{g\}$. Consider the natural isomorphism $\zeta\colon R_g \circ \oot^g \to \oot \circ (R_g \times \id_\C)$, def\/ined as
\begin{gather*}\zeta_{M, V} = \mu_{M, \{g\}, V} \mu^{-1}_{M, {}^g V, \{g\}}\colon \ R_g(M\oot^g V) \to R_g(M) \oot V,
\end{gather*} for all objects $M$ of $\M$ and $V$ of $\C$, where $\mu$ is the associativity constraint of~$\M$.

The functor $R_g$ is an equivalence of categories with quasi-inverse given by the functor $R_{g^{-1}}\colon$ $\M \to \M^g$.

A direct calculation, using the coherence conditions~\eqref{modcat1} and~\eqref{modcat2} for the module category~$\M$, shows that $\zeta$ satisf\/ies conditions~\eqref{uno-z} and~\eqref{dos-z}. Hence $(R_g, \zeta)$ is a $\C$-module functor. Therefore $\M^g \cong \M$ as $\C$-module categories, as claimed. \end{proof}

\begin{Lemma}\label{gdea} Let $H$ be a subgroup of $G$ and let $\psi$ be a $2$-cochain on $H$ satisfying~\eqref{cond-alfa}. Let $A(H, \psi)$ denote the corresponding indecomposable algebra in~$\C$. Then, for all $g \in G$, ${}^gA(H, \psi) \cong A({}^gH, \psi^{g^{-1}} \cch_{g^{-1}})$ as algebras in~$\C$.
\end{Lemma}

\begin{proof} By def\/inition, ${}^gA(H, \psi) = A\big({}^gH, \psi^{g^{-1}} \big(\cch_{g}^{g^{-1}}\big)^{-1}\big)$. It follows from formula~\eqref{rel-omega} that $\big(\cch_{g}^{g^{-1}}\big)^{-1}$ and $\cch_{g^{-1}}$ dif\/fer by a coboundary. This implies the lemma. \end{proof}

\subsection{} Let $H$, $L$ be subgroups of $G$ and let $\psi \in C^2(H, k^\times)$, $\xi \in C^2(L, k^\times)$, be 2-cochains such that $\omega\vert_{H \times H \times H} = d\psi$ and $\omega\vert_{L \times L \times L} = d\xi$.

Let $B$ be an object of the category $_{A(H, \psi)}\C_{A(L, \xi)}$ of $(A(H, \psi), A(L, \xi))$-bimodules in $\C$. For each $z\in G$, let $\pi_l(h)\colon B_z \to B_{hz}$ and $\pi_r(s)\colon B_z \to B_{zs}$, denote the linear maps induced by the actions of $h \in H$ and $s \in L$, respectively. Then the following relations hold, for all $h, h' \in H$, $s, s' \in L$:
\begin{gather}\label{uno}\pi_l(h)\pi_l(h') = \omega(h, h', z) \psi(h, h') \pi_l(hh'), \\
\label{dos}\pi_r(s')\pi_r(s) = \omega(z, s, s')^{-1} \xi(s, s') \pi_r(ss'),\\
\label{tres}\pi_l(h)\pi_r(s) = \omega(h, z, s) \pi_r(s)\pi_l(h).\end{gather}

\begin{Lemma}\label{alfa-g} Let $g \in G$ and let $B_g$ denote the homogeneous component of degree $g$ of $B$. Then the map $\pi\colon H \cap {}^gL \to \text{\rm GL}(B_g)$, defined as $\pi(x) = \pi_r\big({}^{g^{-1}}x\big)^{-1}\pi_l(x)$ is a projective representation of $H \cap {}^gL$ with cocycle~$\alpha_g$ given, for all $x, y \in H \cap {}^gL$, as follows:
\begin{gather*}
\alpha_g(x, y) = \psi(x, y) \xi^{-1}\big({}^{g^{-1}}x, {}^{g^{-1}}y\big) \frac{\omega(x, y, g) \omega\big(x, yg, {}^{g^{-1}}\big(y^{-1}\big)\big)}
{\omega\big(xyg, {}^{g^{-1}}\big(y^{-1}\big),{} ^{g^{-1}}\big(x^{-1}\big)\big)} du_g(x, y) \\
\hphantom{\alpha_g(x, y) =}{} \times \frac{\omega\big({}^{g^{-1} } y, {}^{g^{-1}}\big(y^{-1}\big), {}^{g^{-1}}\big(x^{-1}\big)\big)} {\omega\big({}^{g^{-1}}x, {} ^{g^{-1}}y, {} ^{g^{-1}}\big(y^{-1}x^{-1}\big)\big)},
\end{gather*}
 where the $1$-cochain $u_g$ is defined as $u_g(x) = \omega\big(xg, {} ^{g^{-1}} x, {}^{g^{-1}} \big(x^{-1}\big)\big)$. \end{Lemma}

\begin{proof} It follows from~\eqref{dos} that $\pi_r(s)^{-1} = \omega\big(z, s, s^{-1}\big) \xi\big(s, s^{-1}\big)^{-1} \pi_r\big(s^{-1}\big)$, for all $z\in G$, $s \in L$. In addition, for all $h, h' \in L$, we have the following relation:
\begin{gather*}\xi \big({h'}^{-1}, h^{-1}\big) \xi (h, h') = df(h, h')\frac{\omega\big(h', {h'}^{-1}, h^{-1}\big)}
	{\omega\big(h, h', {h'}^{-1}h^{-1}\big)},\end{gather*}
where $f$ is the 1-cochain given by $f(h) = \xi\big(h, h^{-1}\big)$. A straightforward computation, using this relation and conditions~\eqref{uno},~\eqref{dos} and~\eqref{tres}, shows that $\pi(x) \pi(y) = \alpha_g(x, y) \pi(xy)$, for all $x, y \in H \cap {}^gL$. This proves the lemma. \end{proof}

\begin{Remark}\label{rmk-alfag} Lemma \ref{alfa-g} is a~version of \cite[Proposition~3.2]{ostrik}, where it is shown that $B$ is a~simple object of~$_{A(H, \psi)}\C_{A(L, \xi)}$ if and only if~$B$ is supported on a single double coset $HgL$ and the projective representation~$\pi$ in the component~$B_g$ is irreducible.
\end{Remark}

For all $g \in G$, $\psi^g \cch_g$ is a 2-cochain in $^{g^{-1}\!}\!H$ such that $\omega\vert_{{}^{g^{-1}}H \times {}^{g^{-1}}H \times {}^{g^{-1}} H} = d(\psi^g \cch_g)$. Then the product $\xi^{-1} \psi^g \cch_g$ def\/ines a 2-cocycle of ${}^{g^{-1}} H \cap L$.

\begin{Lemma}\label{rel-cociclos} The class of the $2$-cocycle $\big(\xi^{-1} \psi^g \cch_g\big)^{g^{-1}}$ in $H^2(H\cap {}^gL, k^\times)$ coincides with the class of the $2$-cocycle $\alpha_g$ in Lemma~{\rm \ref{alfa-g}}.
\end{Lemma}

\begin{proof}
A direct calculation shows that for all $x, y \in G$,
\begin{gather*}\frac{\omega\big(y, y^{-1}, x^{-1}\big)}
{\omega\big(x, y, y^{-1}x^{-1}\big)}
\frac{\omega\big({}^gx, {}^gy, g\big) \omega\big({}^gx, {}^gyg, {} y^{-1}\big)}
{\omega\big({}^gx ^gyg, {}y^{-1}, x^{-1}\big)}
= \cch_g(x, y) d\theta_g (x, y),\end{gather*}
where the 1-cochain $\theta_g$ is def\/ined as $\theta_g(x) = \omega\big(g, x, x^{-1}\big)^{-1}$. This implies that $\alpha_g^g \equiv \xi^{-1} \psi^g \cch_g$, as was to be proved.
\end{proof}

\subsection{}\label{demo} In this subsection we give a proof of the main result of this paper.

\begin{proof}[Proof of Theorem \ref{main}] Let $H, L$ be subgroups of $G$ and let $\psi \in C^2(H, k^{\times})$ and $\xi \in C^2(L, k^{\times})$ be 2-cochains satisfying condition~\eqref{cond-alfa}. Let $A(H, \psi)$, $A(L, \xi)$ be the associated algebras in $\C$ and let $\M_0(H, \psi)$, $\M_0(L, \xi)$ be the corresponding $\C$-module categories.

Let $\M = \M_0(L, \xi)$. For every $g \in G$, let $\M^g$ denote the module category induced by the autoequivalence $\ad_g\colon \C \to \C$.
The $\C$-module category $\M^g$ is equivalent to $_{^gA(L, \xi)}\C$. Hence, by Lemma~\ref{gdea}, $\M^g \cong \M_0\big({}^gL, \xi^{g^{-1}} \cch_{g^{-1}}\big)$.

Suppose that there exists an element $g \in G$ such that $H = {}^gL$ and the class of the cocycle $\xi^{-1}\psi^g\cch_g$ is trivial on $L$. Relation~\eqref{rel-omega} implies that $\cch_g^{g^{-1}} = \cch_{g^{-1}}^{-1}$, and thus the class of $\psi^{-1}\xi^{g^{-1}}\cch_{g^{-1}}$ is trivial on~$H$. Then $\psi = \xi^{g^{-1}}\cch_{g^{-1}} df$, for some 1-cochain $f \in C^1(H, k^{\times})$. Therefore $^gA(L, \xi) = A\big(H, \xi^{g^{-1}} \cch_{g^{-1}}\big) \cong A(H, \psi)$ as algebras in $\C$. Thus we obtain equivalences of $\C$-module categories
\begin{gather*}\M_0(L, \xi) \cong \M_0(L, \xi)^g \cong {} _{^gA(L, \xi)}\C \cong \M_0(H, \psi),\end{gather*}
where the f\/irst equivalence is deduced from Lemma~\ref{adj-triv}.

Conversely, suppose that $F\colon \M_0(L, \xi) \to \M_0(H, \psi)$ is an equivalence of $\C$-module categories. Recall that there is an equivalence
\begin{gather*}\label{equiv-fun-1} \Fun_\C\left(\M_0(L, \xi), \M_0(H, \psi)\right) \cong {} _{A(H, \psi)}\C_{A(L, \xi)}.\end{gather*}
Under this equivalence, the functor $F$ corresponds to an object $B$ of $_{A(H, \psi)}\C_{A(L, \xi)}$ such that there exists an object $B'$ of $_{A(L, \xi)}\C_{A(H, \psi)}$ satisfying
\begin{gather}\label{dim-b}B \otimes_{A(L, \xi)}B' \cong A(H, \psi),\end{gather} as $A(H, \psi)$-bimodules in $\C$, and
\begin{gather}\label{bbprime}B' \otimes_{A(H, \psi)}B \cong A(L, \xi), \end{gather} as $A(L, \xi)$-bimodules in $\C$.

Let $\FPdim_{A(H, \psi)}M$ denote the Frobenius--Perron dimension of an object $M$ of $_{A(H, \psi)}\C_{A(H, \psi)}$. Then we have
\begin{gather*}\dim M = \dim A(H, \psi) \FPdim_{A(H, \psi)}M = |H| \FPdim_{A(H, \psi)}M.\end{gather*}
Taking Frobenius--Perron dimensions in both sides of~\eqref{dim-b} and using this relation we obtain that $\dim \left(B \otimes_{A(L, \xi)}B'\right) = |H|$.

On the other hand, $\dim (B \otimes_{A(H, \psi)}B' ) = \frac{\dim B \dim B'}{\dim{A(L, \xi)}} = \frac{\dim B \dim B'}{|L|}$. Thus
\begin{gather}\label{dimbb'} \dim B \dim B' = |H| |L|. \end{gather}

Since $A(H, \psi)$ is an indecomposable algebra in $\C$, then it is a simple object of $_{A(H, \psi)}\C_{A(H, \psi)}$. Then~\eqref{bbprime} implies that $B$ is a simple object of $_{A(H, \psi)}\C_{A(L, \xi)}$ and $B'$ is a simple object of $_{A(L, \xi)}\C_{A(H, \psi)}$.

In view of \cite[Proposition~3.2]{ostrik}, the support of $B$ is a two sided $(H, L)$-double coset, that is, $B = \bigoplus_{(h,h') \in H \times L} B_{hgh'}$, where $g \in G$ is a representative of the double coset that supports~$B$. Moreover, the homogeneous component $B_g$ is an irreducible $\alpha_g$-projective representation of the group $^gL \cap H$, where the 2-cocycle $\alpha_g$ satisf\/ies~$\alpha_g \equiv \big(\xi^{-1}\psi^g \cch_g\big)^{g^{-1}}$; see Remark~\ref{rmk-alfag} and Lemmas~\ref{alfa-g} and~\ref{rel-cociclos}.

Notice that the actions of $h \in H$ and $h'\in L$ induce isomorphisms of vector spaces $B_{g} \cong B_{hg}$ and $B_{g} \cong B_{gh'}$. Hence
\begin{gather}\label{dimb}\dim B = |HgL| \dim B_g = \frac{|H| |L|}{|H \cap {} ^gL|} \dim B_g = [H: H \cap {}^gL] |L| \dim B_g.\end{gather}
In particular, $\dim B \geq |L|$. Reversing the roles of $H$ and $L$, the same argument implies that $\dim B' \geq |H|$. Combined with relations~\eqref{dimbb'} and~\eqref{dimb} this implies
\begin{gather*}|H| |L| = \dim B \dim B' \geq |H| [H : H \cap {}^gL] |L| \dim B_g.\end{gather*}
Hence $[H\colon H \cap {}^gL] \dim B_g = 1$, and therefore $[H\colon H \cap {}^gL] = 1$ and $\dim B_g = 1$. The f\/irst condition means that $H \subseteq {}^gL$, while the second condition implies that the class of $\alpha_g$ is trivial in $H^2(H \cap {}^gL, k^\times)$. Since the rank of $\M_0(H, \psi)$ equals the index $[G:H]$ and the rank of $\M_0(H, \xi)$ equals the index $[G:L]$, then $|H| = |L|$. Thus we get that $H = {}^gL$ and that the class of the 2-cocycle~\eqref{cond-equiv} is trivial in $H^2(L, k^\times)$. This f\/inishes the proof of the theorem.
 \end{proof}

\begin{Example}\label{kp}
Let $B_8$ be the 8-dimensional Kac Paljutkin Hopf algebra. The Hopf algebra $B_8$ f\/its into an exact sequence
\begin{gather*}
k \longrightarrow k^C \longrightarrow B_8 \longrightarrow kL \longrightarrow k,
\end{gather*}
where $C = \mathbb Z_2$ and $L = \mathbb Z_2 \times \mathbb Z_2$. See~\cite{masuoka-6-8}. This exact sequence gives rise to mutual actions by permutations
\begin{gather*}C \overset{\vartriangleleft}\longleftarrow C \times L
\overset{\vartriangleright}\longrightarrow L,\end{gather*}
and compatible cocycles $\tau\colon L \times L \to \big(k^C\big)^\times$, $\sigma\colon C \times C \to (k^L)^\times$, such that $B_8$ is isomorphic to the bicrossed product $kC {}^{\tau}\#_\sigma kL$. The data $\lhd$, $\rhd$, $\sigma$ and $\tau$ are explicitly determined in \cite[Proposition~3.11]{ma-contemp} as follows. Let $C = \langle x\colon x^2 = 1 \rangle$, $L = \langle z, t\colon z^2 = t^2 = ztz^{-1}t^{-1} = 1\rangle$. Then $\lhd\colon C\times L \to C$ is the trivial action of $L$ on $C$, $\rhd\colon C \times L \to L$ is the action def\/ined by $x \rhd z = z$ and $x\rhd t = zt$,
\begin{gather*}\tau_{x^n}\big(z^it^j, z^{i'}t^{j'}\big) = (-1)^{nji'},\end{gather*} for all $0\leq n, i, i', j, j' \leq 1$,
and
\begin{gather*}\sigma_{z^it^j}\big(x^n, x^{n'}\big) = (\sqrt{-1})^{j\big(\frac{n+n'-\langle n+n'\rangle}{2}\big)},\end{gather*}
for all $0 \leq i, j, n, n' \leq 1$, where $\langle n+n'\rangle$ denotes the remainder of $n+n'$ in the division by~$2$. Here we use the notation $\tau(a, a')(y) = : \tau_y(a, a')$ and, similarly, $\sigma(y, y')(a) = : \sigma_a(y, y')$, $a, a' \in L$, $y, y' \in C$.

In view of \cite[Theorem~3.3.5]{schauenburg} (see \cite[Proposition~4.3]{gttic}), the fusion category of f\/inite-dimen\-sio\-nal representations of $B_8^{\rm op} \cong B_8$ is equivalent to the category $\C(G, \omega, L, 1)$, where $G = L \rtimes C$ is the semidirect product with respect to the action~$\rhd$, and $\omega$ is the 3-cocycle arising from the pair $(\tau, \sigma)$ under one of the maps of the so-called \emph{Kac exact sequence} associated to the matched pair.

In this example $G$ is isomorphic to the dihedral group $D_8$ of order 8. The 3-cocycle $\omega$ is determined by the formula
\begin{gather}\label{kac}\omega\big(x^{n}z^{i}t^{j}, x^{n'}z^{i'}t^{j'}, x^{n''}z^{i''}t^{j''}\big) =
\tau_{x^n}\big(z^{i'}t^{j'}, x^{n'} \rhd z^{i''}t^{j''}\big) \sigma_{z^{i''}t^{j''}}\big(x^n, x^{n'}\big),\end{gather}
for all $0 \leq i, j, i', j', i'', j'', n, n', n'' \leq 1$.

Notice that $\omega\vert_{L \times L\times L} = 1$. Hence, for every 2-cocycle $\xi$ on $L$, the pair $(L, \xi)$ gives rise to an indecomposable $\C$-module category $\M(L, \xi)$.
Formula~\eqref{kac} implies that $\cch_x\vert_{L \times L}$ is given by
\begin{gather*}\cch_x\big(z^{i}t^{j}, z^{i'}t^{j'}\big) = (-1)^{ji'}, \qquad 0\leq i, i', j, j' \leq 1.\end{gather*}
Then $\cch_x$ is a 2-cocycle representing the unique nontrivial cohomology class in $H^2(L, k^\times)$. By Theorem~\ref{main}, for any 2-cocycle~$\xi$ on $L$, $\M_0(L, 1)$ and $\M_0(L, \xi)$ are equivalent as $\C(G, \omega)$-module categories, and therefore so are the corresponding $\C$-module categories $\M(L, 1)$ and $\M(L, \xi)$. This implies that indecomposable $\C$-module categories are in this example parameterized by conjugacy classes of subgroups of $D_8$ on which $\omega$ has trivial restriction, as claimed in \cite[Section~6.4]{MN}.
\end{Example}

\subsection*{Acknowledgements}

This research was partially supported by CONICET and SeCyT - Universidad Nacional de C\'ordoba, Argentina.

\pdfbookmark[1]{References}{ref}
\LastPageEnding

\end{document}